\documentclass[a4paper,12pt]{amsart}
\usepackage{amssymb}
\usepackage{amsmath}
\usepackage[applemac]{inputenc}

\def\.{\cdot}

\def\beq{\begin{equation}}
\def\eeq{\end{equation}}
\def\bea{\begin{eqnarray*}}
\def\eea{\end{eqnarray*}}
\def\beaa{\begin{eqnarray}}
\def\eeaa{\end{eqnarray}}
\def\ba{\begin{array}}
\def\ea{\end{array}}
\def\f{\varphi}

\def\bp{\begin{proof}}
\def\r{\end{proof}}

\def \RM{\mathbb{R}}

\def\Ric{\mathrm{Ric}}
\def\id{\mathrm{id}}
\def\be{\begin{equation}}
\def\ee{\end{equation}}
\def\tr{\mathrm{tr}}

\def\so{\mathfrak{so}}

\def\SU{\mathrm{SU}}
\def\U{\mathrm{U}}

\def\SO{\mathrm{SO}}

\def\Sp{\mathrm{Sp}}
\def\Spin{\mathrm{Spin}}

\def\scal{\mathrm{scal}}

\def\Id{\mathrm{id}}

\def\II{\mathrm{II}}
\def\vol{\mathrm{vol}}

\newtheorem{epr}{Proposition}[section]

\newtheorem{ath}[epr]{Theorem}

\newtheorem{elem}[epr]{Lemma}

\theoremstyle{definition}

\newtheorem{ere}[epr]{Remark}

\title{Extrinsic hyperspheres in manifolds with special holonomy}

\author{Tillmann Jentsch, Andrei Moroianu and Uwe Semmelmann}

\address{Tillmann Jentsch\\
Institut f\"ur Geometrie und Topologie \\
Fachbereich Mathematik\\
Universit{\"a}t Stuttgart\\
Pfaffenwaldring 57 \\
70569 Stuttgart, Germany
}
\email{tillmann.jentsch@mathematik.uni-stuttgart.de}

\address{Andrei Moroianu \\ CMLS\\ {\'E}cole Polytechnique \\ UMR 7640 du CNRS
\\ 91128 Palaiseau \\ France}
\email{am@math.polytechnique.fr}

\address{Uwe Semmelmann\\
Institut f\"ur Geometrie und Topologie \\
Fachbereich Mathematik\\
Universit{\"a}t Stuttgart\\
Pfaffenwaldring 57 \\
70569 Stuttgart, Germany
}
\email{uwe.semmelmann@mathematik.uni-stuttgart.de}

\date{\today}

\begin{document}

\begin{abstract}
We describe extrinsic hyperspheres and totally geodesic hypersurfaces  in
manifolds with special holonomy. In particular we prove the nonexistence 
of extrinsic hyperspheres in quaternion-K\"ahler manifolds. We develop a
new approach to extrinsic hyperspheres based on the classification of
special Killing forms. 
\bigskip


\end{abstract}

\maketitle

\section{Introduction}

A submanifold of a Riemannian manifold is called an extrinsic sphere if it is totally umbilical and has non-zero
parallel mean curvature vector field. This concept was introduced by Nomizu and Yano in~\cite{nomizu} as
a natural analogue to ordinary spheres in Euclidean spaces. Extrinsic spheres have been studied  intensively 
during the last fifty years. In several cases it was shown that extrinsic spheres have to be Euclidean spheres
and partial classifications were obtained. 

In this article we will only consider the case of extrinsic hyperspheres,
i.e. extrinsic spheres of codimension one. In general, the  existence of extrinsic 
hyperspheres seems to impose strong restrictions on the geometry of the ambient manifold, e.g. Chen and 
Nagano~\cite{chen-nagano}  showed that a locally irreducible symmetric space admitting an extrinsic hyperspheres 
has  to be of constant curvature. However there are also interesting examples of extrinsic spheres which are
not isometric to ordinary spheres. In particular Sasakian manifolds appear as extrinsic hyperspheres of
K\"ahler manifolds (cf.~\cite{yamaguchi}). Hence it is natural to  ask for  the existence of extrinsic hyperspheres 
in manifolds with {\it special} holonomy, i.e. manifolds whose restricted holonomy group 
is strictly contained in the corresponding special orthogonal group.
By the Berger-Simons holonomy theorem, we have to consider the following cases: The manifold can be locally
a Riemannian product, a locally symmetric space or its restricted holonomy group is one of 
$\U(m),\ \SU(m),\ \Sp(m)$, $\Sp(m)\cdot \Sp(1),\ {\mathrm G}_2$ or $\Spin(7)$. Our first result concerns 
quaternion-K\"ahler manifolds, i.e. Riemannian manifolds with restricted holonomy contained in
$\Sp(m)\cdot \Sp(1)$, with $m\ge 2$. We prove the following

\begin{ath}\label{qk}
A quaternion-K\"ahler manifold of non-vanishing scalar curvature does
not admit an extrinsic hypersphere. 
\end{ath}

Similarly we show that there are no extrinsic hyperspheres in complete
Riemannian products without one-dimensional factors and in complete
manifolds with  
holonomy ${\mathrm G}_2$ or $\Spin_7$. However we give non-complete
examples, as metric cones over manifolds with special geometric
structures, such as Sasakian, nearly K\"ahler or nearly  
parallel ${\mathrm G}_2$-structures. In fact every manifold is an
extrinsic hypersphere in its (non-complete) metric cone. 

Our main observation in the proof is that a parallel form on the ambient manifold naturally defines a so-called
special Killing form on any extrinsic hypersphere. Then we use the classification of special Killing forms
(cf.~\cite{uwe}) and in particular the fact that these forms define parallel forms on the metric cone.
This gives a unified approach to the investigation of extrinsic hyperspheres, which also reproves 
some of the known results in the local product and K\"ahler case.

At some points we also use a remarkable theorem of Koiso (cf.~\cite{koiso} or Theorem~\ref{koiso}). In particular 
this theorem states that a complete Einstein manifold of non-constant sectional curvature does not admit any 
extrinsic hypersphere which is itself Einstein and has positive scalar curvature. 
As a striking consequence we note that, contrary to the general expectation, it is not possible to construct 
new examples of 6-dimensional 
nearly K\"ahler manifolds as totally umbilical hypersurfaces of complete nearly parallel ${\mathrm G}_2$-manifolds.

Finally we consider totally geodesic  hypersurfaces in manifolds with special holonomy. 
We first show that the problem of finding totally geodesic hypersurfaces in a locally reducible 
manifold reduces to the same problem for one  of the locally defined factors, see Theorem~\ref{ath:pr}.
Our main result in the irreducible case is then the following

\begin{ath}\label{ath:tg}
There do not exist
any totally geodesic hypersurfaces in 
\begin{enumerate}
\item  locally irreducible K\"ahler-Einstein manifolds (including  Calabi-Yau and  hyperk\"ahler manifolds);
\item  Quaternion-K\"ahler manifolds  
\item  manifolds with holonomy  ${\mathrm G}_2$ or $\Spin(7)$.
\item  locally irreducible symmetric spaces of non-constant sectional
  curvature. 
\end{enumerate}
\end{ath}

In particular, Theorems \ref{qk}  and \ref{ath:tg} imply that a complete quaternion-K\"ahler manifold does not admit any (possibly non-complete) 
totally umbilical hypersurface.

\section{Preliminaries}
Let  $(\bar M, \bar g)$ be an  $(n+1)$-dimensional Riemannian manifold
and let $i: M \subset \bar M$ be a
submanifold with induced Riemannian metric $g$. The second fundamental form is defined as
$
\II(X,Y) = \bar \nabla_XY - \nabla_XY
$
where $X$ and $Y$ are vector fields tangent to $M$ and $\nabla$ resp. $ \bar\nabla$ denote the
Levi-Civita connections of $g$ resp. $ \bar g$.  Let $N$ be a normal vector field on $M$ then the shape operator 
$A_N X := (\bar \nabla_XN)^T$ is related to the second fundamental form via
$$
\bar g (\II(X, Y), N) = \bar g (A_N X, Y) \ ,
$$
for any vector fields $X,Y$ on $M$. A submanifold $ M \subset \bar M$  is said to be {\it totally umbilical} if
$\II(X,Y) = g(X,Y)H$, with $H=\frac1n \tr \, \II $ denoting the {\it mean curvature} vector field of $M$ in $\bar M$.
Choosing a  parallel unit length normal vector field $N$, this condition can be written as 
$\II (X,Y) = \lambda \, g(X,Y) \, N$ for some function $\lambda$ on $M$.
The manifold $M$ is called {\it totally geodesic } in $\bar M$ if the equation $\II = 0$ holds, corresponding to the special
case $\lambda = 0$.

In this article we are especially interested in {\it extrinsic hyperspheres},  i.e. complete hypersurfaces such that
$\II (X,Y) = \lambda g(X,Y) N$ for some real constant $\lambda\ne 0$. 

Let  $ M \subset \bar M$ be a totally umbilical hypersurface, with unit length normal vector field $N$,
then  the covariant derivative $\bar \nabla$ may  be written as
\beq\label{eh}
\bar \nabla_XY = \nabla_XY + \lambda g(X,Y)N,\qquad \bar\nabla_XN = -\lambda X \ ,
\eeq
where $X,Y$ denote vector fields tangent to $M$. For totally umbilical hypersurfaces the 
curvature equations of Gau\ss{} and Codazzi take the following form:

$$
\begin{array}{rl}
\bar R(X,Y,Z,W) & = \quad R(X,Y,Z,W) \;+\;  \lambda^2 g(X \wedge Y, Z \wedge W) \\[2ex]
\bar R (X,Y,Z,N) & = \quad X(\lambda) \,g(Y, Z) \;-\; Y(\lambda)\,g(X,Z)
\;=\; (d\lambda \wedge Z) (X, Y)
\end{array} 
$$

\noindent
where $X,Y,Z,W$ are vector fields on $M$, and $\bar R $ resp. $R$ denote the Riemannian curvature tensors
of $\bar g$ resp. $ g$, and $Z$ is identified with its dual 1-form
using the metric $g$. Let the curvature operator $R$ on 2-vectors be defined by 
$$
g(R(X\wedge Y), Z \wedge W) = - R(X,Y,Z,W) \ ,
$$
so that the curvature operator of the standard sphere is the identity. Then the Gau\ss{} equation may also be written as
$
R = \bar R + \lambda^2 \,\id \  .
$
Hence the sectional curvatures $\bar K$ resp.  $K$ of $\bar g$ resp. $ g$ are related by
$K = \bar K + \lambda^2$.

The following well-known lemma will be helpful below (cf. \cite{koiso}). 

\begin{elem}\label{factor}
Let $(M^n, g), n \ge 2,$ be a  totally umbilical hypersurface of an Einstein manifold $(\bar M, \bar g)$. Then
$\II = \lambda g N$, for some constant $\lambda$ and a parallel unit length normal vector field $N$,
i.e. complete totally umbilical  hypersurfaces in Einstein manifolds are extrinsic spheres. Moreover,
$$
\lambda^2 \;=\;  \frac{\scal_g}{n(n-1)}  \; - \;   \frac{\scal_{\bar g}}{n(n+1)} \ 
$$
and  $\scal_g$ is constant. In particular the inequality $(n+1) \, \scal_g \ge (n-1) \, \scal_{\bar g}$ holds, with equality 
in the case of a totally geodesic hypersurface, i.e. for $\lambda = 0$.
\end{elem}
\proof
Let $M \subset \bar M$ be a totally umbilical hypersurface in $\bar M$. By definition we have
$\II = \lambda g N$ for some function $\lambda$ on $M$.
Let $\{e_i\}, i=1, \ldots, n+1$, with $e_{n+1}:=N$, be a  local orthonormal frame for $T\bar M$ restricted to $M$.
Then  the Ricci curvature $\overline \Ric$ of $\bar g$ applied to a vector field $X$ tangent to $M$ can 
be computed using the  Gau\ss{} equation:

$$
\begin{array}{rl}
\overline\Ric(X,X) & = \sum^n_{i=1} \bar R(X,e_i, e_i, X)   \;+\; \bar R(X, N, N, X) \\[1.5ex]
& = \sum^n_{i=1} (R(X,e_i,e_i, X) \;+\; \lambda^2 [g(X, e_i)^2 - g(X,X) g(e_i, e_i)] ) \;+\; \bar R(X, N, N, X) \\[1.5ex]
& = \Ric(X,X) \;-\; \lambda^2 (n-1) |X|^2 \;+\; \bar R(X, N, N, X)
\end{array}
$$

\noindent
In this equation we take the trace over an orthonormal base in $TM$ and use the assumption that $(\bar M, \bar g)$
is Einstein to obtain
$$
n \, \frac{   \scal_{\bar g} }{n+1} \;\;=\;\; \scal_g \;-\;
n(n-1)\lambda^2 \;+\; \frac{\scal_{\bar g}}{n+1}  
$$

\noindent
This proves the equation for $\lambda^2$, the inequality and the
characterization of the case of 
equality. It remains to show that $\lambda$, and thus also $\scal_g$,
is constant. This immediately 
follows from the Codazzi equation. Indeed if we take the trace over a
local orthonormal frame on $M$ 
we obtain for any vector field $X$ on $M$
$$
\overline{\Ric}\,(X, N) \;=\; (n-1)\, d\lambda(X) \.
$$
Hence, $d\lambda = 0$ and we conclude that  $\lambda$ as well as
$\scal_g$ have to be constant on $M$.

\qed

\begin{ere} 
Using a result of~\cite{JR} on the existence of submanifolds with parallel 
second fundamental form (e.g. totally umbilical submanifolds with 
$\lambda=const$), one can show that in a complete manifold
with real analytic metric every (possibly non-complete) submanifold with parallel 
second fundamental is contained in a complete one. Further, we recall
that every Einstein metric is real analytic with respect to normal coordinates according to 
a theorem of DeTurck and Kazdan, see~\cite{koiso}.  
It follows that in a complete Einstein manifold every totally
umbilical submanifold is an open part of an extrinsic sphere. 
\end{ere}

In general we will not assume that the ambient manifold $\bar M$ has
to be complete. However if we assume 
completeness, as well as the Einstein condition for $g$ and $\bar g$,
the following theorem of Koiso 
(cf. \cite{koiso}) gives a rather strong restriction for extrinsic
hyperspheres.

\begin{ath}[Koiso]\label{koiso}
Let $(M, g)$ be a totally umbilical Einstein hypersurface in a
complete Einstein manifold $(\bar M, \bar g)$. 
Then the only possible cases are:

(a)  $g$ has positive Ricci curvature. Then $g$ and $\bar g$ have
constant sectional curvature 

(b) $\bar g$ has negative Ricci curvature. If $\bar M$ is compact or
$(\bar M, \bar g)$ homogeneous, then $g$ and $\bar g$ 
      have constant sectional curvature
      
(c) $g$ and $\bar g$ have zero Ricci curvature. If $(\bar M, \bar g )$
is simply connected, then $(\bar M, \bar g )$ 
decomposes as $(\tilde M, \tilde g) \times \RM $, where $(\tilde M,
\tilde g)$ is a totally geodesic hypersurface 
in $( \bar M, \bar g) $ which contains $M$.
\end{ath}

\section{Extrinsic hyperspheres}
In this section we will study totally umbilical submanifolds in ambient spaces with special holonomy.

\subsection{Special Killing forms}
Let  $i: M \subset \bar M$ be an $n$-dimensional extrinsic hypersphere
in a  manifold $(\bar M, \bar g)$ 
with special holonomy. Except for the case of symmetric spaces, the
restriction of holonomy is directly 
linked to the existence of certain parallel differential forms $\sigma
\in \Omega^k(\bar M)$. The main tool 
in our investigation of extrinsic hyperspheres is the observation that
the pull-back forms  
$i^*(N \lrcorner \,  \sigma)$  and $i^*\sigma$ are special Killing
resp. $\ast$-Killing forms on $M$ 
(cf. \cite{uwe}). Here and henceforth $N$ denotes a unit normal vector
field along $M$. 

\begin{elem}\label{restrict}
Let  $i: M \subset \bar M$ be an $n$-dimensional extrinsic hypersphere and  
$\sigma$ be a non-trivial parallel $k$-form on $\bar M$ and let $\gamma := i^*(N \lrcorner \,  \sigma)$  and 
$\beta := i^*\sigma$ be the pull-back forms on $M$. Then for every vector field $X$ on $M$ the following
equations hold:
$$
\begin{array}{lrl}
(i) &
\nabla_X \gamma  & =  \qquad   \frac{1}{k}\, X \lrcorner \, d \gamma    \\[1.5ex]
(ii) & 
\nabla_X d \gamma  & =  \quad -\,  k  \lambda^2 \,X \wedge \gamma  \\[1.5ex]
(iii) &
 \nabla_X \beta & =  \quad  -\, \frac{1}{n-k+1}\, X \wedge d^* \beta  \\[1.5ex]
(iv) &
 \nabla_X d^* \beta  & =  \quad  (n-k+1 ) \lambda^2 \; X \lrcorner \,  \beta,
\end{array}
$$
where the non-zero constant $\lambda$ is given by \eqref{eh}.
In particular, it follows that $\gamma$ is coclosed and $\beta$ is closed. Moreover, the
forms $\gamma$ and $\beta$ are related by
$$
 d \gamma   =   -\, k\,\lambda \,  \beta, \qquad d^*\beta = - \, (n-k+1)\lambda \, \gamma \ .
$$
Furthermore, $\gamma$ is a non-parallel $k-1$-form on
$M$.

\end{elem}

\proof
Let  $i: M \rightarrow \bar M$ the inclusion map of the extrinsic hypersphere $M$, then
the differential $i_*$ identifies $T_pM$ with a subspace of $T_p\bar M$. We have
$i_* (\nabla_XY) = \bar \nabla_XY - \lambda g(X,Y) N$, where $X,Y$ are vector fields
tangent to $M$. Let $X, X_1,\ldots, X_k$ be vector fields on $M$ then
$$
\begin{array}{rl}
(\nabla_X i^*\sigma)(X_1, \ldots, X_k)
& =
X(\sigma(X_1,\ldots, X_k)) - \sum_j \sigma ( \ldots, i_*(\nabla_XX_j), \ldots) \\[1.5ex]
& =
(\bar \nabla_X \sigma)(X_1, \ldots, X_k) \; + \; 
\lambda \, \sum_j g(X, X_j)\,\sigma(\ldots, N, \ldots) \\[1.5ex]
& =
\lambda \, (X \wedge i^*[N \lrcorner \, \sigma])(X_1, \ldots, X_k)
\end{array}
$$
It immediately follows that $i^*\sigma $ is closed. Moreover, contracting  with $X=e_k$ and 
summing over a local orthonormal base $\{e_k\}$ of $TM$ yields  
$
d^*(i^*\sigma ) = - (n-k+1)\,\lambda\, i^*[N \lrcorner \, \sigma] \ . 
$
Substituting this into the equation for $\nabla_X i^*\sigma$ proves (iii). Similarly we find
$$
\nabla_X i^*(N \lrcorner \, \sigma) \;=\; \bar \nabla_X(N \lrcorner \, \sigma)
\;=\; (\bar \nabla_XN) \lrcorner \,  i^*\sigma
\;=\; -\,\lambda \, X \lrcorner \,i^* \sigma \  .
$$
This implies that  $i^*(N \lrcorner \, \sigma)$ is coclosed and that 
$
di^*(N \lrcorner \, \sigma) = -k\,\lambda\, i^*\sigma \ ,
$
completing also the proof of equation (i). Finally we use the calculations above to 
conclude the proof of equations (ii) and (iv):
$$
\begin{array}{rl}
\nabla_X d\gamma  & =\; -\,k\,\lambda\,\nabla_X\beta \;=\; -\,k\,\lambda^2\, X \wedge \gamma,
\\[1.5ex]
\nabla_X d^*\beta &  =\; -\,(n-k+1)\,\lambda \nabla_X\gamma
\;=\; (n-k+1)\, \lambda^2\, X \lrcorner \, \beta \ .
\end{array}
$$

Suppose, by contradiction, that $\gamma$ is a
parallel $k-1$-form on $M$. Then $d\gamma=0$, hence
$\beta=-\frac{1}{k\,\lambda}d\gamma=0$ and thus
$\gamma=-\frac{1}{(n-k+1)\,\lambda}d^*\beta=0$.
Therefore, $\sigma|_M=0$ which is not possible since $\sigma$ is a non-trivial
parallel $k$-form on $\bar M$. We conclude that $\gamma$ is not parallel.
\qed

Equations (i) and (ii) define a non-parallel special $(k-1)$-Killing form $\gamma$.
Complete manifolds admitting such forms were classified in \cite{uwe}. It turns out
that special Killing forms  can only exist on Euclidian spheres, Sasakian- and  3-Sasakian manifolds, nearly
K\"ahler manifolds in dimension 6 or nearly parallel $G_2$-manifolds in dimension 7.

The classification is based on the fact that every special Killing
$(k-1)$-form $\psi$ defines a  
parallel $k$-form $\tilde \psi$ on the metric cone $\tilde M$,
i.e. the manifold 
$\tilde M = M \times \RM_+$ with the cone metric $\tilde g = t^2 g + dt^2$. 
Recall that the metric cone is a non-complete manifold, which contains
the complete manifold $M$ as an extrinsic hypersphere. The parallel form on
$\tilde M$ is defined as  
$
\tilde \psi = \frac{1}{k} d(t ^{k} \psi) = t^{k-1} dt \wedge \psi  \,
+  \,\frac1k \,t^{k}d\psi
$ 

It is important to note that this construction assumes a certain
normalization in equations 
(ii) and (iv), which in our case is equivalent to $\lambda^2 = 1$.
Clearly, after a constant rescaling of the
metric $\bar g$ and replacing $N$ with $-N$ if
necessary, one can even assume that $\lambda=-1$
in \eqref{eh}. 

\noindent
{\bf Remark: }{\it  After finishing our paper we were informed about the article \cite{raulot},
where a part of Lemma~\ref{restrict} is proved independently. The authors show that the restriction of a parallel form onto a
extrinsic hypersphere defines a special Killing form.
}

Returning to our situation, let $i: M \subset \bar M$ be an $n$-dimensional
extrinsic hypersphere and $\sigma$ be a non-trivial parallel
$k$-form on $\bar M$. Restricted to
the submanifold $M \subset \bar M$, we may write $\sigma$ with the notation from above as
$$
\sigma \;=\; N \wedge i^*(N \lrcorner \, \sigma) \, +\,  i^*\sigma\; = \;N
\wedge \gamma + \beta \, .
$$
where $\gamma$ is  a non-parallel special $(k-1)$-
Killing form such that $d\gamma = k \beta$. Thus we obtain that 
$$
\tilde \sigma \;=\; t^{k-1} dt \wedge \gamma \;+\; t^k \,\beta \ ,
$$
is a non-trivial parallel $k$-form on the cone $\tilde M$. Obviously, 
the $k$-forms $\sigma$ and $\tilde \sigma$ 
have the same algebraic type. This implies that
their stabilizer under the $\SO(n+1)$-action on $\Lambda^k$ has to be the same.

It is well known that if $(M,g)$ is complete, the metric cone $\tilde
M$  has reducible holonomy only if it is flat, in which case $M$ is
isometric to the standard sphere \cite{gallot}. Moreover, if the cone metric $\tilde g$ is Einstein,
then it has to be Ricci flat. In particular, the metric cone can be symmetric only if it is flat. 
Indeed an irreducible symmetric space is Einstein, thus the cone is then Ricci flat and also flat.
Similarly, the metric cone can not be a quaternion-K\"ahler manifold, since these 
manifolds are automatically Einstein. Thus the scalar curvature of the cone vanishes and the
holonomy is reduced to $\Sp(m)$, i.e. the cone is in fact  hyperk\"ahler.  
According to the the Berger list, there 
remain five cases of irreducible cones  $\tilde M$ admitting  parallel forms: K\"ahler, Calabi-Yau,  hyperk\"ahler
manifolds, and manifolds with holonomy ${\mathrm G}_2$ resp. $\Spin(7)$, in dimensions $7$ resp $8$.
It follows that $M$ has a Sasakian, Einstein-Sasakian, 3-Sasakian, 
nearly K\"ahler or nearly parallel ${\mathrm G}_2$ structure, respectively.

\subsection{Quaternion-K\"ahler manifolds}
Let $(\bar M^{4m}, \bar g)$ be a quaternion-K\"ahler manifold, i.e. a Riemannian manifolds with (restricted) holonomy  contained in $\Sp(m)\cdot \Sp(1)$. Since for $m=1$ the holonomy 
condition is empty, one usually assumes $m\ge 2$. On quaternion-K\"ahler  manifolds
one has a parallel $4$-form $\sigma$, the so-called Kraines form. Its stabilizer is the
group  $\Sp(m)\cdot \Sp(1) \subset \SO(4m)$. 

Let $M \subset \bar M$ be an extrinsic hypersphere. Thus it admits a
special Killing form 
and carries one of the special geometric structures mentioned
above. Let us first assume 
that $M$ is Sasakian, but not Einstein. Then the cone $\tilde M$ is an
irreducible 
K\"ahler manifold with holonomy equal to $\U(2m)$. The parallel forms
are powers of the K\"ahler form, whose stabilizers  
contain $\U(2m)$. But for $m\ge 2$ the unitary group  $\U(2m)$ is not
contained in  $\Sp(m)\cdot \Sp(1) $.  
Thus this case is not possible.

In the remaining cases, $M$ is the standard sphere, Einstein-Sasakian,
3-Sasakian, nearly K\"ahler, or nearly parallel 
${\mathrm G}_2$, and is Einstein with positive scalar
curvature $\scal_g = n(n-1)$.
If $\bar M$ would be complete then we could apply the result of Koiso,
i.e. Theorem~\ref{koiso},  
to rule out these cases. However, even if $\bar M$ is not complete, we
may exclude the 
remaining possibilities.
Indeed, the cone over $M$ has to be Ricci flat and Lemma~\ref{factor}
shows that 
$\scal_{\bar g}= 0$. Thus the holonomy of $(\bar M, \bar g)$ reduces
further to $\Sp(m)$, which is a different case. 

This  proves  Theorem~\ref{qk}.

\qed

\subsection{K\"ahler manifolds}
This case also includes Calabi-Yau and hyperk\"ahler manifolds.
It is well known that a K\"ahler form $\sigma \in \Omega^2(\bar M)$
induces a Sasakian structure on any  
extrinsic hypersphere $M \subset \bar M$. The Killing vector field of
the Sasakian structure 
is given by $\xi = JN = N \lrcorner \, \sigma$
(cf. \cite{yamaguchi}). This situation was also studied in \cite{chen2} and \cite{chen3}. Non-complete examples are 
obtained as metric cones over Sasakian, Einstein-Sasakian resp. 3-Sasakian
manifolds. 

However we do not know of any example of a complete K\"ahler manifold
admitting an extrinsic hypersphere.

\subsection{Manifolds with holonomy ${\mathrm G}_2$ or $\Spin(7)$}
Let $(\bar M, \bar g)$ be a manifold with holonomy contained in $ {\mathrm G}_2$ or $\Spin(7)$.
Then $\bar M$ carries a parallel $3$- resp. $4$-form $\sigma$ and the 
$2$- resp. $3$-form $N \lrcorner \, \sigma$ defines a nearly K\"ahler resp. nearly parallel
${\mathrm G}_2$-structure on any extrinsic hypersphere $M \subset \bar M$. These manifolds are
Einstein with positive scalar curvature and we may use the result of Koiso from Theorem~\ref{koiso}
to exclude them as hypersurfaces of complete manifold $\bar M$. 

Again there are non-complete examples $\bar M$, as metric cones over nearly K\"ahler resp.
nearly parallel ${\mathrm G}_2$-manifolds. Conversely it follows from~\cite[Eq.~(2.3.b)]{koiso} that any
$\bar M$ with holonomy ${\mathrm G}_2$ or $\Spin(7)$ admitting an extrinsic hypersphere $M$ is 
locally isometric to the cone over $M$.

More generally, it is an old and well known observation of Gray~\cite{gray} that already
the existence of a nearly parallel ${\mathrm G}_2$-structure on $\bar M^7$ implies the existence
of a nearly K\"ahler structure on any totally umbilical hypersurface $M\subset \bar M$.
However, as we have just seen, it is a striking consequence of  Koiso's Theorem~\ref{koiso}
that if $\bar M$ is complete, then no new examples of nearly K\"ahler
manifolds can be produced in this way.

\subsection{Local product manifolds}
Let $\bar g$ be a Riemannian product metric on $\bar M = \bar M_1 \times \bar M_2$. We assume
that not both factors have dimension one. The volume forms $\vol_{M_1}, \vol_{M_2}$ are
parallel forms on $\bar M$. Let $\bar M_1$ be the factor with a non-vanishing projection of the normal
vector $N$. Then $N \lrcorner \, \vol_{\bar M_1}$ is different from zero and defines, as described above, 
a parallel form of degree $\dim \bar M_1$ on the cone over $M$. However this form has a
non-trivial kernel (the vectors from $T\bar M_2 \cap TM$), which is at least one-dimensional
and defines a parallel distribution on the cone. Hence the cone is reducible and, if M is
complete, we apply the theorem of Gallot~\cite{gallot} to conclude that the cone is flat and
$M$ is isometric to the sphere. This result was first obtained  by M. Okumura~\cite{okumura} 
(using the Obata Theorem).


Finally we remark that if we do not require the completeness condition
for $M$ and $\bar M$, then the situation is much more flexible and
one can construct lots of examples by taking $(\bar M,\bar g)$ to be a
product of two Riemannian cones $(M_1\times\RM,t^2g_1+dt^2)$ and
$(M_2\times\RM,s^2g_2+ds^2)$. Indeed, such a product is always a  
Riemannian cone over the manifold $M=M_1\times M_2\times \RM$ endowed
with the incomplete Riemannian metric
$g:=\sin^2\theta\,g_1+\cos^2\theta\,g_2+d\theta ^2$, 
as shown by the formula (cf. \cite{mo})
$$(t^2g_1+dt^2)+(s^2g_2+ds^2)=r^2(\sin^2\theta\,g_1+\cos^2\theta\,g_2+d\theta
^2)+dr^2, \qquad (s,t)=(r\cos\theta,r\sin\theta).$$
The manifold $(M,g)$ is thus embedded as a totally umbilical hypersurface in $(\bar M,\bar g)$.

\subsection{Locally symmetric spaces}
Extrinsic spheres in locally symmetric spaces are well understood. It
follows from results of  
Chen~\cite{chen1} that the real space forms are the only irreducible
locally symmetric spaces admitting 
extrinsic hyperspheres. Since every locally irreducible symmetric space 
is a complete Einstein manifold, this result is also implied
by  Theorem~\ref{koiso} of Koiso. Moreover, any extrinsic hypersphere
in a symmetric space is a symmetric 
submanifold  in the sense of~\cite[Ch.~9.3]{bco}
(cf.~\cite[Proposition~9.3.1]{bco}). Therefore, if $\bar M$ is a product
$\bar M_1\times\cdots\times \bar M_k$ where $\bar M_i$ are simply connected irreducible
symmetric spaces, it follows from a result of Naitoh~\cite{naitoh} 
that any extrinsic hypersphere is of the form $M_1\times \bar M_2\times\cdots\times \bar M_k$ where $M_1$ 
is an extrinsic hypersphere in a space $\bar M_1$ of constant curvature.

Similar results are true for certain classes of homogeneous spaces. In \cite{tojo} Tojo proves
that compact normally homogeneous spaces admitting extrinsic hyperspheres have 
constant sectional curvature.  The same conclusion is proved by Tsukada in \cite{tsukada}
for isotropy irreducible homogeneous spaces admitting totally umbilical hypersurfaces.


\section{Totally geodesic hypersurfaces}
\label{sec:hypersurfaces}

\newcommand{\frakg}{\mathfrak{g}}
\newcommand{\frakp}{\mathfrak{p}}
\newcommand{\frakk}{\mathfrak{k}}
\newcommand{\rmS}{\mathrm{S}}

There are many  examples of totally geodesic  hypersurfaces in (possibly non-complete) Einstein manifolds.
In fact Koiso proves in \cite{koiso} the following

\begin{ath}\label{koiso2}
Let $(M, g)$ be a real analytic Riemannian manifold with constant scalar curvature.
Then there exists a (possibly non-complete) Einstein manifold $(\bar M, \bar g)$ such that $(M, g)$ is 
isometrically embedded into $(\bar M, \bar g)$ as a totally geodesic hypersurface.
Moreover, such $(\bar M, \bar g)$ is essentially uniquely determined. More precisely, 
if $(\tilde M, \tilde g)$ is a second Einstein manifold which contains $M$ as a totally geodesic hypersurface, 
then there exist open neighborhoods $\bar U$ and $\tilde U$ 
of $M$ in $\bar M$ and $\tilde M$, respectively, and an isometry $I:\tilde U\to \bar U$ with $I|_M=\Id$.
\end{ath}

In this section we will show that the Einstein manifold $\bar M^{n+1}$
given by Koiso's Theorem~\ref{koiso2} can never have special holonomy if
$(M,g)$ is locally irreducible. In fact, 
then $\bar M$ is locally irreducible, too, because of Theorem~\ref{ath:pr}. Thus we can apply Theorem~\ref{ath:hol} 
in order to obtain that the restricted holonomy group of $\bar M$ is given by $\SO(n+1)$.

\subsection{Local Products}
We will first show that if $(\bar M,\bar g)$ is locally reducible and complete, 
then the problem of finding totally geodesic hypersurfaces in $\bar M$ reduces to the same problem on one of the factors. 
More precisely, we will prove the following:

 \begin{ath}\label{ath:pr}
 Let $(\bar M,\bar g)$ be a complete, simply connected manifold with reducible holonomy, and assume that 
 $(M,g)$ is a complete totally geodesic hypersurface of $\bar M$. 
Then $\bar M$ can be written as a Riemannian product $(\bar M,\bar g)=(\bar M_1,\bar g_1)\times (\bar M_2,\bar g_2)$ 
such that $M$ is equal to $M'_1\times \bar M_2$, where $M'_1$ is a complete totally geodesic hypersurface of $\bar M_1$.
 \end{ath}
 
 \begin{proof} Since $\bar M$ is complete, simply connected and has reducible holonomy, 
the de Rham decomposition theorem shows that it 
is isometric to a Riemannian product $\bar M = \bar M_1 \times \bar M_2$, with $\bar g=\bar g_1+\bar g_2$. The exponential function clearly satisfies
 \beq\label{exp}
 \exp_{(x_1,x_2)}^{\bar M}(X_1,X_2)=(\exp_{x_1}^{\bar M_1}(X_1),\exp_{x_2}^{\bar M_2}(X_2))\eeq
 for all $(x_1,x_2)\in\bar M$ and $(X_1,X_2)\in T_{(x_1,x_2)}\bar M$.
 
 Let $M \subset \bar M$ be a totally geodesic hypersurface with unit length normal vector field $N$. 
 With respect to the decomposition 
$T\bar M = T \bar M_1 \oplus T \bar M_2$, the vector field $N$ can be written as $N = X_1 + X_2$ at every point of $M$. If at some point
$x =(x_1,x_2)\in M$ one component, e.g. $X_1$, vanishes, then $T_xM =T_{x_1}
\bar M_1\times X_2^\perp$, 
thus by \eqref{exp} $M=\bar M_1\times
\exp_{x_2}(X_2^\perp)$, where the second factor is clearly a totally geodesic
hypersurface in $\bar M_2$. In the following we will assume that both components of $N$ are different from zero.

For every $x =(x_1,x_2)\in M$ we write $N_x=aN_1+bN_2$ where $a$ and $b$ are functions on $M$ and $N_1$, $N_2$ 
are unit vectors in $T_{x_1}\bar M_1$ and $T_{x_2}\bar M_2$ depending a priori on $x_2$ and $x_1$ respectively. 
We will show later on that they actually do not depend on these variables.

Let $\omega_i$ denote the restriction to $M$ of the volume forms of the two factors of $\bar M$. 
Consider the vector field $H$ on $M$ defined by
$g(H,\cdot)=*(*\omega_1\wedge*\omega_2)$ (the Hodge dual $*$ is that of
$M$). Up to a sign, depending on the orientation of $M$, one has
$H=bN_1-aN_2$. Since $\omega_i$ are parallel, $H$ 
is a parallel vector field on $M$, so if $\f_t$ denotes its flow, then $f_t(x)$ is a geodesic for all $x\in M$. 

Let us fix some $x =(x_1,x_2)\in M$ and consider 
the totally geodesic surfaces $M_1=(\bar M_1\times \{x_2\})\cap M$ and
$M_2=(\{x_1\}\times \bar M_2)\cap M$ of $\bar M_1$ and $\bar M_2$
respectively. The projection of $T_xM$ to $T_{x_1}\bar M_1$ is onto, therefore the projection $\pi_1:M\to \bar M_1$ is onto. Indeed, 
for every $y_1\in \bar M_1$ there exists $Y_1\in T_{x_1}\bar M_1$ such that $y_1=\exp_{x_1}^{\bar M_1}(Y_1)$, 
so by \eqref{exp}, $y_1=\pi_1(\exp_{(x_1,x_2)}^{\bar M}(Y_1,Y_2))$, where $Y_2$ is chosen so that $(Y_1,Y_2)\in T_xM$.

We will now show that $aN_1$ only depends on $x_1$. Indeed, the set of $y_2\in\bar M_2$ such that 
$(x_1,y_2)\in M$ is just $M_2$, and for every vector $Y_2\in T_{x_2}M_2$ we have $\bar g(Y_2,N)=0$, so 
$0=\nabla^{\bar g}_{Y_2}N=\nabla^{\bar g}_{Y_2}(aN_1)+\nabla^{\bar
  g}_{Y_2}(bN_2)$. Since the two terms 
in the right hand factor are tangent to $\bar M_1$ and $\bar M_2$ respectively, they both vanish. 
In particular, $\nabla^{\bar g}_{Y_2}(aN_1)=0$, and since $N_1$ has unit length, $a$ and $N_1$ are both constant along $M_2$. 
This fact, together with the previous observation that the projections of $M$ on $\bar M_1$ and $\bar M_2$ are onto, 
show that there exist globally defined functions $a$ on $\bar M_1$, $b$ on
$\bar M_2$ and vector fields $\bar N_1$ on $\bar M_1$, $\bar N_2$ on $\bar
M_2$, such that $N_x=a(x_1)\bar N_1(x_1)+b(x_2)\bar N_2(x_2)$ for all $x=(x_1,x_2)\in M$.

We claim that $\bar N_1$ is parallel on $\bar M_1$. First, if $X\in
T_{x_1}M_1$, 
then $X$ is orthogonal to $N$ so $\nabla_X(a\bar N_1)=0$ like before. 
Since $\bar N_1$ has unit length, this shows that $X(a)=0$, so $a$ is constant
along $M_1$ and $\nabla_X\bar N_1=0$. 
It remains to check the parallelism in the direction of $\bar N_1$ itself. 
Equation \eqref{exp} shows that the geodesic $\exp_{x_1}(tb\bar N_1)$ is the
projection in $\bar M_1$ of $\exp_x(tH)$, whose tangent vector at every $t$ is
$H$. Thus the tangent vector of $\exp_{x_1}(tbN_1)$ is
$(bN_1)(\exp_{x_1}(tbN_1))$, showing that $bN_1$ is parallel in the direction of $N_1$. On the other hand, 
we have already seen that $b$ only depends on the second variable, so $N_1$ is parallel at $x_1$, and thus everywhere on $\bar M_1$. 
Similarly, $N_2$ is parallel on $\bar M_2$. By the de Rham theorem again, one can write $\bar M_1=M_1\times \RM$, $\bar g_1=g_1+dt^2$, 
$N_1=\partial/\partial t$ and $\bar M_2=M_2\times \RM$, $\bar g_2=g_2+ds^2$, $N_2=\partial/\partial s$. From the above, the functions $a$ and $b$ only depend on $t$ and $s$ 
respectively, but since $a^2+b^2=1$, they are both constant. This shows that identifying $\bar M$ with $\RM\times (M_1\times M_2\times
\RM)$ 
by the isometry $((t,x_1),(s,x_2))\mapsto (at+bs,(x_1,x_2,bt-as))$, $N$ is identified to the unit tangent vector to the $\RM$-factor, 
and thus $M$ is isometric to the second factor $M_1\times M_2\times \RM$. This finishes the proof of the theorem.
\end{proof}

\subsection{Irreducible manifolds with totally geodesic hypersurfaces}
\label{sec:irreducible_hypersurfaces}

Now we turn our attention to the case where $(\bar M,\bar g)$ is
locally irreducible. 
Recall that a hypersurface $M$ of a Riemannian manifold $(\bar M,\bar
g)$ is called {\it locally reflective} 
if the geodesic reflection $r$ in $M$ defines an isometry of a
suitable open neighborhood $U$ of $M$ in $\bar M$. 
Then $r$ is locally given by $r(\exp(t N_p))=\exp(-t N_p)$ (where
$N_p$ denotes the 
normal vector at $p \in M$). Moreover, we recall that a locally
reflective submanifold is  
automatically totally geodesic (cf. \cite{bco}).

If $M$ is a totally geodesic hypersurface of an Einstein manifold
$(\bar M,\bar g)$, then $M$ has constant scalar curvature 
according to Lemma~\ref{factor}.
In this situation, N. Koiso has shown that $M$ is a locally reflective
submanifold, cf. Remark~7 of~\cite{koiso}. 

\begin{ath}\label{ath:hol}
Let $(\bar M,\bar g)$ be a locally irreducible Riemannian manifold. 
If there exists an $n$-dimensional locally reflective hypersurface  $M
\subset \bar M$,  
then the restricted holonomy group of $\bar M$ is equal to $\SO(n+1)$. In particular, there are no 
totally geodesic hypersurfaces in locally irreducible Einstein manifolds with special holonomy. 
\end{ath}
\proof

Let $U$ be an open neighborhood of $M$ in $\bar M$ in which the geodesic reflection $ r $ in $M$ is defined. 
Clearly, it suffices to prove the theorem in case $U=\bar M$. Since $r(p)=p$ for all points $p \in M$, 
we obtain an involutive Lie group homomorphism 
$\tau:\SO(T_p\bar M)\to \SO(T_p\bar M)$  which is given by $\tau(g)=d_p r\circ g\circ d_p r$. 
Since $d_p r$ is the linear reflection in $T_pM$ and the normal vector $N_p$ 
spans the whole normal space at $p$, the connected component 
of the fixed point group under $\tau$ is equal to $\SO(T_pM)$.  
Further, let $G$ denote the restricted holonomy group of $\bar M$ 
at the point $p \in M$. Then for every closed, null-homotopic curve $\alpha:[0,1]\to \bar M$ 
the curve $r\circ\alpha$ is again closed and null-homotopic. If $g$ denotes the parallel displacement along $\alpha$, 
then the parallel displacement along $r\circ\alpha$ is given by $\tau(g)$ (because $r$ is an isometry of $\bar M$).
We obtain that $\tau(g)\in G$ for all $g\in G$. 
Let $H$ denote the subgroup of $G$ which is fixed under $\tau$ and $H_0$ be its
connected component. Set $\tilde H:= H\cap\SO(T_pM)$, then $H_0\subset\tilde H\subset H$, hence $(G,\tilde H)$ is a Riemannian symmetric pair in the
sense of \cite[Ch.~IV, \S 3]{helgason}.
 In particular, any $G$-invariant metric makes  $G/\tilde H$  a Riemannian symmetric space. Moreover, there is a natural injective map $\iota:G/\tilde H\to \rmS^{n}$ 
which is given by $[g]\mapsto g(N_p)$, where $\rmS^n$  is considered as  the Euclidian sphere of $T_p\bar M$. 
We claim that $\iota$ is a totally geodesic map, i.e. $\iota$ maps geodesics of $G/\tilde H$ into geodesics of 
$S^n$:

Let $\frakp:=\{x\wedge N_p \, |\,  x\in T_pM\}$ be the Cartan complement of 
$\frakk:=\so(T_pM)$ in $\so(T_p\bar M)$.
Then $\so(T_p\bar M)=\frakk\oplus\frakp$ with  $d_p\tau(A)=A$ 
for all $A\in\frakk$ and $d_p\tau(A)=-A$ for all $A\in \frakp$. 
Let $\frakg$  denote the Lie algebra of $G$, then the Cartan decomposition of $\frakg$ is given by 
$(\frakk\cap\frakg)\oplus(\frakp \cap \frakg)$. Let $\gamma$ be a geodesic
of $G/\tilde H$ through the origin $\tilde H$. Then there exists some $A\in\frakp\cap\frakg$ such that 
$\gamma(t)=[\exp(t A)]$ (cf.~\cite{bco}). Therefore, $\iota(\gamma(t))=\exp(t A)N_p$, which is a geodesic line of $\rmS^{n}$.

This shows that the dimension $k$ of the symmetric space $G/\tilde H$ is less or equal $n$ 
and the orbit $GN_p$ is a $k$-dimensional totally geodesic submanifold of $\rmS^n$, i.e. 
$GN_p$ is a standard Euclidian sphere $\rmS^k\subset\rmS^n$. 
Then the linear subspace of $T_p\bar M$ which is spanned by $GN_p$ is $G$-invariant and hence $k=n$, since $G$
acts irreducibly on $T_p\bar M$. It follows that $\dim(\frakp\cap\frakg)=\dim(G/\tilde H)=n=\dim(\frakp)$ and thus $\frakp\subset \frakg$,
therefore $\so(T_p\bar M)=[\frakp,\frakp]\oplus\frakp\subset\frakg$. We obtain 
that actually $\frakg=\so(T_p\bar M)$. Switching from Lie algebras  to Lie groups, we conclude that the 
connected component of $G$ is equal to $\SO(T_p\bar M)$. The result now follows.

\qed

This also proves Theorem~\ref{ath:tg}, since all ambient manifolds in question are Einstein with special holonomy.

\qed


\labelsep .5cm

\end{document}